\documentclass[12pt]{amsproc}
\usepackage{amssymb}

\usepackage{amsmath}

\newtheorem{thm}{Theorem}

\theoremstyle{remark}

\numberwithin{equation}{section}

\newcommand{\la}{\langle}
\newcommand{\ra}{\rangle}

 \newcommand{\al}{\alpha}

\def\cA{\hbox{$\mathcal A$}}

\def\cN{\hbox{$\mathcal N$}}

\def\cO{\hbox{$\mathcal O$}}

\def\fS{\hbox{$\mathfrak S$}}

\def\bN{\hbox{$\mathbb N$}}

\newcommand{\LO}{\mathcal{LO}}

\def\ds{\hbox{${}^{**}$}}

\def\L1{\hbox{$L^1(G)$}}
\def\M{\hbox{$M(G)$}}
\def\WS{\hbox{$WAP(S)$}}
\def\WG{\hbox{$WAP(G)$}}

\def\dt{\hbox{$\delta_t$}}

\def\dtm{\hbox{$\delta_{t^{-1}}$}}

\begin{document}
J. Math. Anal. Appl. to appear

\title[group algebras]{$2m$-Weak amenability of group algebras} 

\author{Yong Zhang}

\address{ Department of Mathematics\\ 
University of Manitoba\\ Winnipeg MB R3T 2N2\\ Canada}
\email{zhangy@cc.umanitoba.ca\hfill \break
}

\thanks{Supported by NSERC 238949-2011. }

\begin{abstract}
A common fixed point property for semigroups is applied to show that the group algebra \L1 of a locally compact group $G$ is $2m$-weakly amenable for each integer $m\geq 1$.
\end{abstract}

\def\timestring{\begingroup
\count0=\time \divide\count0 by 60
\count2 = \count0
\count4 = \time \multiply\count0 by 60
\advance\count4 by -\count0
\ifnum\count4 <10 \toks1={0}%
\else \toks1 = {}%
\fi
\ifnum\count2<12 \toks0 = {a.m.}%
\else \toks0 = {p.m.}%
\advance\count2 by -12
\fi
\ifnum\count2=0 \count2 = 12
\fi
\number\count2:\the\toks1 \number\count4
\thinspace \the\toks0
\endgroup}

\def\date{ \timestring \enskip {\ifcase\month\or January\or February\or
March\or April\or May\or June\or July\or August\or September\or October\or
November\or December\fi}
\number\day}

\maketitle
\begin{center}
\today
\end{center}

\section{Introduction}

Let \cA\  be a Banach algebra and $X$ a Banach $\cA$-bimodule. A linear mapping $D$: $\cA\to X$ is called a \emph{derivation} if it satisfies $D(ab) = aD(b) + D(a)b$ ($a,b\in \cA$). Given any $x\in X$, the mapping $\text{ad}_x$: $a\mapsto ax-xa$ ($a\in \cA$) is a continuous derivation, called an \emph{inner derivation}. 

If $X$ is a Banach $\cA$-bimodule, then the dual space $X^*$ of $X$ is naturally a Banach $\cA$-bimodule with the $\cA$-module actions defined by
\[ \la x, af\ra = \la xa,f\ra \quad \la x, fa\ra = \la ax,f\ra \quad(a\in \cA, f\in X^*, x\in X).\]
Note that the Banach algebra \cA\ itself is a Banach \cA-bimodule with the product giving the module actions. So $\cA^{(n)}$, the $n$-th dual space of \cA, is naturally a Banach \cA-bimodule in the above sense for each $n\in \bN$. The Banach algebra $\cA$ is called \emph{$n$-weakly amenable} if every continuous derivation from $\cA$ into $\cA^{(n)}$ is inner. If $\cA$ is $n$-weakly amenable for each $n\in \bN$ then it is called \emph{permanently weakly amenable}.

Let $G$ be a locally compact group. The integral of a function $f$ on a measurable subset $K$ of $G$ against a fixed left Haar measure is denoted by $\int_K{f}dx$. Two functions on $G$ are regarded identical if they are equal to each other almost everywhere with respect to the left Haar measure. The group algebra \L1\ is the Banach algebra consisting of all absolutely integrable functions on $G$ (with respect to the left Haar measure), equipped with the convolution product and the usual $L^1$ norm 
\[\|f\|_1 := \int_G{|f(t)|}dt. \]
When $G$ is discrete, $\L1$ is $\ell^1(G)$ consisting of all absolutely summable functions on $G$.

 B.E. Johnson showed in \cite{Joh_weak} 
that \L1\ is always $1$-weakly amenable for any locally compact group $G$. 
It was shown further in \cite{D-G-G} that  \L1 is in fact $n$-weakly amenable for all odd numbers $n$. Whether this is still true for even numbers $n$ was left open in \cite{D-G-G}. For a free group $G$, Johnson proved later in \cite{Joh 1999} that  $\ell^1(G)$ is indeed $2m$-weakly amenable for any $m\in \bN$. The problem has been resolved affirmatively for general locally compact group $G$ in \cite{C-G-Z} and in \cite{Losert10} independently, using a theory established in \cite{Losert08}.

In this note we present a short proof to the $n$-weak amenability of \L1\ for even numbers $n$. Our proof is based on a common fixed point property for semigroups. In Section~\ref{FP} we study this fixed point property.  For the general theory concerning  amenability and fixed point properties of locally compact groups we refer the reader to \cite{greenleaf, Ric}. The proof to the main result will be given in Section~\ref{proof}. 

\section{Common fixed points for semigroups}\label{FP}

Let $S$ be a (discrete) semigroup. The space of all bounded complex valued functions on $S$ is denoted by $\ell^\infty(S)$. It is a Banach space with the uniform supremum norm. In fact $\ell^\infty(S) = (\ell^1(S))^*$, the dual space of $\ell^1(S)$.  For each $s\in S$ and each $f\in \ell^\infty(S)$ let $\ell_sf$ be the left translate of $f$ by $s$, that is $\ell_sf(t)=f(st)$ ($t\in S$) (the right translate $r_sf$ is defined similarly). A function $f\in \ell^\infty(S)$ is called \emph{weakly almost periodic} if its left orbit $\LO(f) = \{\ell_sf:\, s\in S\}$ is relatively compact in the weak topology of $\ell^\infty(S)$. The space of all weakly almost periodic functions on $S$ is denoted by $WAP(S)$, which is a closed subspace of $\ell^\infty(S)$ containing the constant function and invariant under the left and right translations. A linear functional $m\in WAP(S)^*$ is a \emph{mean} on $WAP(S)$ if $\|m\| = m(1) =1$. A mean $m$ on $WAP(S)$ is a \emph{left invariant mean} (abbreviated LIM) if $m(\ell_sf) = m(f)$ for all $s\in S$ and all $f\in WAP(S)$. If $S$ is a group, it is well known that $WAP(S)$ always has a LIM \cite{greenleaf}.

Let $X$ be a Banach space and $C$ a nonempty subset of $X$. A mapping $T$: $C\to C$ is called \emph{nonexpansive} if $\|T(x)-T(y)\|\leq \|x-y\|$ for all $x,y\in C$. When $X$ is a separable locally convex topological space whose topology is determined by a family $Q$ of seminorms on $X$, we  will denote it by $(X,Q)$ to highlight the topology $Q$.

Let $C$ be a subset of a locally convex topological vector space $(X,Q)$. We say that $\fS = \{T_s:\; s\in S\}$ is a \emph{representation} of $S$ on $C$ if for each $s\in S$, $T_s$ is a mapping from $C$ into $C$ and $T_{st}(x) = T_s(T_tx)$ ($s,t\in S$, $x\in C$). 
The representation is called continuous if each $T_s$ ($s\in S$) is $Q$-$Q$ continuous; It is called equicontinuous if for each neighborhood $\cN$ of $0$ there is a neighborhood $\cO$ of $0$ such that $T_s(x) -T_s(y)\in \cN$ whenever $x,y\in C$, $x-y \in \cO$ and $s\in S$. The representation is called \emph{affine} if $C$ is convex and each $T_s$ ($s\in S$) is an affine mapping, that is  $T_s(ax+by) = aT_s(x) + bT_s(y)$ for all constants $a,b \geq 0$ with $a +b = 1$, $s\in S$ and $x, y\in C$. We say that $x\in C$ is a \emph{common fixed point} for (the representation of) $S$ if $T_s (x) = x$ for all $s\in S$. 

The following fixed point theorem was proved in \cite{Lau76}.
\begin{thm}\label{lau}
Let $S$ be a discrete semigroup and $\fS$ an equicontinuous affine representation of $S$ on a weakly compact convex subset $C$ of a separated locally convex space $X$. If $WAP(S)$ has a left invariant mean then $C$ contains a common fixed point for $S$.
\end{thm}

Let $B$ be a nonempty bounded subset of a Banach space $X$. By definition the \emph{Chebyshev radius} of $B$ in $X$ is
\[  r_B =\inf\{r\geq 0: \exists x\in X\; \sup_{b\in B}\|x-b\| \leq r\}. \]
Clearly we have $0\leq r_B < \infty$ and 
\begin{equation}\label{radius}
\sup_{b\in B}\|x-b\| \geq r_B\quad \text{for each } x\in X.
\end{equation}
 The \emph{Chebyshev center} of $B$ in $X$ is defined to be
\[  C_B =\{x\in X: \sup_{b\in B}\|x-b\| \leq r_B\}.  \]
Chebyshev center has been extensively used in the field of fixed point theory (see \cite{D-S,G-K}). Some asymptotic version of it has been employed to study fixed point properties of semigroups \cite{Lim 80, Lim 74, Lim Pac}.

We now recall that a Banach space $X$ is \emph{L-embedded} if the image of $X$ under the canonical embedding into its bidual $X^{**}$, still denoted by $X$, is an $\ell_1$ summand in $X^{**}$, that is if there is a subspace $X_s$ of $X^{**}$ such that $X^{**} = X\oplus_1 X_s$, where $\oplus_1$ denotes the $\ell_1$ direct sum. The class of L-embedded Banach spaces includes all $L^1(\Sigma,\mu)$ (the space of of all absolutely integrable functions on a measure space $(\Sigma,\mu)$), preduals of von Neumann algebras, dual spaces of M-embedded Banach spaces and the Hardy space $H_1$. In particular, given a locally compact group $G$, the space $L^1(G)$ is L-embedded. So are its even duals $L^1(G)^{(2m)}$ ($m\in \bN$). We refer to \cite{HWW} for more details of the theory concerning this type of Banach spaces. We also refer to  \cite{BGM, BPP, Jap 02} for the study of fixed points of various mappings in an L-embedded Banach space. In \cite{BGM}, as an application of a fixed point theorem, a surprising short solution to the well-known derivation problem was given. The problem was first settled by V. Losert in \cite{Losert08}. 

 We now give a common fixed point theorem for semigroups, which will provide the major machinery for our proof to the main result.
\begin{thm}\label{L-embedded}
Let $S$ be a discrete semigroup and $\fS$ a representation of $S$ on an L-embedded Banach space $X$ as nonexpansive affine mappings. Suppose that \WS\ has a LIM and suppose that there is a nonempty bounded set $B\subset X$ such that $B\subseteq \overline{T_s(B)}$ for all $s\in S$, then $X$ contains a common fixed point for $S$.
\end{thm}
\begin{proof}
We use the idea of \cite{BGM} to show that there is a nonempty weakly compact convex set in $X$ that is $S$-invariant. We first regard $B$ as a subset of $X\ds$. Let $r_B$ be the Chebyshev radius and $C$  the Chebyshev center of $B$ in $X\ds$. Then $C$ is nonempty, weak* compact and convex. In fact, for each $r> r_B$, 
\[ 
C _r :=\{x\in X\ds: \, \sup_{b\in B}\|x-b\|\leq r\} 
\]
 is nonempty by the definition of $r_B$. Note that $C_r = \cap_{b\in B}B[b,r]$, where $B[b,r]$ denotes the closed ball in $X\ds$ centered at $b$ with radius $r$. The set $C_r$ is convex and weak* compact since each $B[b,r]$ is. The collection $\{C_r :\, r> r_B\}$ is decreasing as $r$ decreases. Thus $C = \cap_{r> r_B}C_r$ is nonempty and is still weak* compact and convex. By the L-embeddedness of $X$ there is a subspace $X_s$ of $X\ds$ such that $X\ds = X \oplus_1X_s$. Let $x\in C$. then there are $c\in X$ and $\xi\in X_s$ such that $x = c +\xi$. For each $b\in B$, $\|x-b\| = \|c-b\| + \|\xi\|$. So
\[ r_B \geq \sup_{b\in B}\|x-b\| = \sup_{b\in B}\|c-b\| +\|\xi\|\geq r_B +\|\xi\|. \]
The last inequality is due to (\ref{radius}).
 Therefore, we must have $\xi = 0$. This shows that $ C\subset X$. The weak* compactness of $C$ (in $X\ds$) is the same as the weak compactness of it (in $X$). So $C$ is a nonempty, weakly compact and convex subset of $X$.

Now for $s\in S$, $b\in B$ and $x\in C$ we have 
\[ \|T_s(x) - T_s(b)\| \leq \|x- b\| \leq r_B \]
since $T_s$ is nonexpansive. This implies that $\|T_s(x) - a\| \leq r_B$ for $a\in \overline{T_s(B)}$ ($s\in S$, $x\in C$). In particular, this holds for all $a\in B$ since $B\subseteq \overline{T_s(B)}$. Thus
 $T_s(x) \in C$ whenever $x\in C$ and $s\in S$, showing that $C$ is $S$-invariant. Note that a nonexpansive representation of $S$ is indeed equicontinuous. By Theorem~\ref{lau}, there is a common fixed point for $S$ in $C$. The proof is complete.
\end{proof}

Theorem~\ref{lau} has been extended to the general semitopological semigroup setting in \cite{L-Z}. A more general version of Theorem~\ref{L-embedded} and some discussion on when there is a set $B$ such that $T_s(B) =B$ for all $s\in S$ can also be found there.

\section{$2m$-weak amenability of \L1}\label{proof}
 
 Let $X$ be a Banach space. Denote the space of all bounded linear operators on $X$ by $B(X)$. The space $B(X)$ is a Banach algebra with the operator norm topology and the composition product. So is $B(X)\times B(X)^\text{op}$ with the product topology and coordinatewise operations, where $B(X)^\text{op}$ is the algebra formed by reversing the order of the product in $B(X)$. The \emph{strong operator topology} (or briefly $so$-topology) on $B(X)\times B(X)^\text{op}$ is the topology induced by the family of seminorms $\{p_x: x\in X\}$, where
\[  p_x(S, T) = \max\{\|S(x)\|, \|T(x)\|\} \quad (S,T\in B(X)) \]
(see \cite[page 327]{DAL}). 

Given a locally compact group $G$, let $M(G)$ be the space of all bounded complex valued regular Borel measures on $G$. With the convolution product of measures and with the norm induced by the total variation, $M(G)$ is a Banach algebra containing $L^1(G)$ as a closed ideal.  In fact, $\M$ is the multiplier algebra of $\L1$, 
and as the multiplier algebra of $\L1$, $\M$ is a subalgebra of $B(\L1) \times B(\L1)^{\text{op}}$ with each $\mu\in \M$ being identified with (the double multiplier) $(\ell_\mu, r_\mu) \in B(\L1) \times B(\L1)^\text{op}$, where $\ell_\mu$ and $r_\mu$ denote, respectively, the left multiplier operator and the right multiplier operator on $L^1(G)$ implemented by $\mu$.  We refer to \cite{DAL} for the standard theory about multipliers and multiplier algebras.

It is well-known that  $lin\{\dt: t\in G\}$, the linear space generated by the point measures $\dt$ ($t\in G$), is dense in $\M$ in the $so$-topology \cite[Proposition~3.3.41(i)]{DAL}. In particular, for each $h\in \L1$ there is a net $(u_\al)\subset lin\{\dt: t\in G\}$ such that $\|(u_\al - h)*a\|_1 \to 0$ and $\|a*(u_\al - h)\|_1 \to 0$ for all $a\in \L1$.

Recall that if $\cA$ is a Banach algebra, then its bidual $\cA^{**}$ is a Banach algebra equipped with the Arens product $\Box$ defined 
\[
\la f, u\Box v\ra = \la v\cdot f, u\ra, \quad v\cdot f \in \cA^*: \;\la a, v\cdot f\ra = \la fa, v\ra
\] 
for $u,v\in \cA^{**}$, $f\in \cA^*$ and $a\in \cA$. If $X$ is a Banach $\cA$-bimodule, then its bidual $X^{**}$ is naturally a Banach $\cA^{**}$-bimodule with the module actions given by

\[ \la F, u\cdot M\ra = \la M\cdot F, u\ra, \quad M\cdot F \in A^* : \; \la a, M\cdot F\ra = \la F\cdot a, M\ra \] 
and
\begin{align*}
 &\la F,  M\cdot u\ra = \la u\cdot F, M\ra, \quad u\cdot F \in X^* :\; \la x, u\cdot F\ra = \la F\cdot x, u\ra, \\
 &F\cdot x \in A^*: \;\la a, F\cdot x\ra = \la x\cdot a, F\ra  
\end{align*} 
for $u\in \cA^{**}$, $M\in X^{**}$, $F\in X^*$, $x\in X$ and $a\in \cA$. In particular,  for any integer $m\in\bN$, $\cA^{(2m)}$ is a Banach $\cA^{**}$-bimodule.

A Banach $\cA$-bimodule $X$ is called \emph{neo-unital} if $X = \cA X\cA$, that is every element $x\in X$ may be written in the form $x=ayb$ for some $a,b\in \cA$ and $y\in X$. If $\cA$ has a bounded approximate identity $(e_\al)$ and $X$ is a neo-unital Banach $\cA$-bimodule, then we may extend the $\cA$ bimodule actions on $X$ to $M(\cA)$, the multiplier algebra  of $\cA$. The extension is defined as follows.
\[
\mu x = \lim_\al (\mu e_\al)x = (\mu a)yb, \quad x\mu = \lim_\al x( e_\al\mu) = ay(b\mu)
\]
for $\mu\in M(\cA)$ and $x=ayb\in X$. Here we note that $\mu a, b\mu \in \cA$ since $\cA$ is (always) an ideal of $M(\cA)$.
These operations make $X$ a unital Banach $M(\cA)$-bimodule. In this case a continuous derivation $D$: $\cA \to X^*$ may be extended to a continuous derivation from $M(\cA)$ to $X^*$ by defining
\[
D(\mu) = \text{wk*-}\lim_\al D(\mu e_\al) \quad (\mu\in M(\cA)).
\]
Moreover this extended $D$ is $so$-weak* continuous. In fact, if $\mu_\al \to \mu$ in $M(\cA)$ in the $so$-topology and $x=ayb \in X$ for $a,b\in \cA$ and $y\in X$, then
\begin{align*}
\lim_\al \la x,D(\mu_\al) \ra &= \lim_\al \la ay,D(b\mu_\al) \ra - \lim_\al \la \mu_\al ay,D(b) \ra \\
                              &=  \la ay,D(b\mu) \ra -  \la \mu ay,D(b) \ra = \la x,D(\mu) \ra .
\end{align*}
We refer to the seminar paper \cite{JOH} and the monograph \cite{DAL} for more details of the above extensions.

We now can prove the main result of the paper.
\begin{thm}\label{2m L1}
Let $G$ be a locally compact group. Then the group algebra $L^1(G)$ is $2m$-weakly amenable for each $m\in \bN$.
\end{thm}

\begin{proof}
Denote $\cA = L^1(G)$, $X = \cA^{(2m)}$ and $Y = \cA^{(2m-1)}$. Then, as we have indicated, $X$ is a Banach $\cA^{**}$-bimodule. 
Let $(e_\al)$ be a bounded approximate identity of $\cA$ and let $E$ be a weak* cluster point of $(e_\al)$ in $\cA^{**}$. Then $Ea = aE = a$ for all $a\in \cA$. We have the $\cA$-bimodule decomposition
 $X = X_1 \oplus X_2 \oplus X_3$, where 
 \[
 X_1 = \ell_E\circ r_E (X), \quad X_2 = (I-r_E)(X), \quad X_3 =(I-\ell_E)\circ r_E(X).
 \] 
 Here $I$ denotes the identity operator, $\ell_E$ is the left multiplication by $E$ and $r_E$ the right multiplication by $E$.
It is readily seen that 
 \[
 X_2 = (\cA Y)^\perp \cong (Y/\cA Y)^*,\quad X_1 \oplus X_3 = r_E(X) \cong (\cA Y)^*
 \]
 as Banach $\cA$-bimodules. Similarly, in $(\cA Y)^*$
 \[
  (I-\ell_E)\left((\cA Y)^*\right) = (\cA Y\cA)^\perp \cong (\cA Y/\cA Y\cA)^*
 \]
and 
\[
\ell_E\left((\cA Y)^*\right) \cong (\cA Y \cA)^*.
 \]
as Banach $\cA$-bimodules. We have 
\[
X_3 \cong (\cA Y/\cA Y\cA)^* \quad
 \text{and } \;
X_1 \cong (\cA Y \cA)^*.
\]

Let $D$: $\cA \to X$ be a continuous derivation. Then $D = D_1 + D_2 + D_3$, where
\begin{align*}  D_1 &= \ell_E\circ r_E\circ D: \cA \to X_1, \quad D_2 = (I-r_E)\circ D: \cA \to X_2,\\
  D_3 &=(I-\ell_E)\circ r_E\circ D: \cA \to X_3 .\end{align*}
Since $\ell_E$ and $r_E$ are $\cA$-bimodule morphisms, $D_1, D_2$ and $D_3$ are continuous derivations. Note that the left $\cA$-module action on $Y/\cA Y$ and the right $\cA$-module action on $\cA/\cA Y\cA$ are trivial. 
From \cite[Proposition~1.5]{JOH}, $D_2$ and $D_3$ are inner. 
We now show that $D_1$ is also inner. Then $D$ must be inner. 

Since $\cA Y\cA$ is neo-unital, we may extend $D_1$ to a continuous derivation from $M(G)$, the multiplier of $\cA$, to $X_1$.
So we may consider $\Delta$: $G \to X_1\subset X$ defined by 
\[ \Delta(t) = D_1(\dt)\cdot \dtm \quad (t\in G). \]
It is readily seen that 
\begin{equation}\label{delta}
\Delta(ts) =  \dt\cdot \Delta(s) \cdot\dtm + \Delta(t) \quad (t,s\in G).
\end{equation}
Let $B = \Delta(G)$. Then $B$ is a nonempty bounded subset of $X$.
For each $t\in G$, let $T_t$ be the self mapping on $X$ defined by 
\[ T_t(x) = \dt\cdot x \cdot\dtm + \Delta(t) \quad (x\in X). \] 
Using (\ref{delta}) one may check that $\fS = \{T_t: t\in G\}$ defines a representation of $G$ on $X$ which is clearly nonexpansive and affine.
Moreover, $T_t(\Delta(s)) = \Delta(ts)$ ($t,s\in G$) and $T_e =I$. Since $G$ is a group, the above implies $T_t(B) = B$ for each $t\in G$. Here $G$ is regarded as a discrete group.

Since \WG\ has a LIM and $X$ is L-embedded, by Theorem~\ref{L-embedded}, there is $\xi \in X$ such that 
\[  \dt\cdot \xi\cdot\dtm + \Delta(t) = \xi \quad \text{for all } t\in G. \] 
So $D_1(\dt) = \xi\cdot\dt - \dt\cdot \xi = \text{ad}_{-\xi}(\dt)$ ($t\in G$). Let $x=\ell_E\circ r_E(-\xi)$. Then $x\in X_1$. Also $D_1(\dt)\in X_1$. For any $ayb \in \cA Y\cA$ with $a,b\in \cA$ and $y\in Y$, we have
\begin{align*}
 \la ayb,D_1(\dt)\ra & = \la ayb\cdot\dt- \dt\cdot ayb,-\xi\ra = \la E(ayb\cdot\dt- \dt\cdot ayb)E,-\xi\ra \\
                               &= \la ayb\cdot\dt- \dt\cdot ayb, x\ra = \la ayb, \text{ad}_x(\dt)\ra \quad t\in G. 
\end{align*}
So it is true that $D_1(\dt) = \text{ad}_{x}(\dt)$ for all $t\in G$.
From what we have shown before stating the current theorem, both $D_1$ and $\text{ad}_{x}$, as continuous derivations from $\M$ into the dual of a neo-unital $\cA$-bimodule, are $so$-weak* continuous. 
Since $lin(\dt: t\in G)$ is dense in $M(G)$ in the $so$-topology, we finally have 
\[  D_1(f) = \text{ad}_{x}(f) \quad (f\in \cA = L^1(G)), \]
therefore $D_1$ is inner.
The proof is complete.
\end{proof}

\end{document}